\documentclass[11pt]{amsart}
\usepackage{amsmath, amssymb, amscd, mathrsfs, url, pinlabel,verbatim}
\usepackage[pagebackref]{hyperref}
\usepackage[margin=1.25in,marginparwidth=1in,centering,letterpaper,dvips]{geometry}
\usepackage{color,dcpic,latexsym,graphicx,epstopdf,comment}
\usepackage[all]{xy}
\usepackage[dvipsnames]{xcolor}
\usepackage{tikz,pgfplots}
\usepackage{enumitem,upquote,tabularx,textcomp}
\usepackage[all]{hypcap}
\usepackage[color=blue!20!white,textsize=tiny]{todonotes}

\title{An instanton take on some knot detection results}

\author[John A. Baldwin]{John A. Baldwin}
\address{Department of Mathematics \\ Boston College}
\email{john.baldwin@bc.edu}

\author[Steven Sivek]{Steven Sivek}
\address{Department of Mathematics\\Imperial College London}
\email{s.sivek@imperial.ac.uk}
\address{Max Planck Institute for Mathematics}
\email{ssivek@mpim-bonn.mpg.de}

\thanks{JAB was supported by NSF FRG Grant DMS-1952707.}

\makeatletter
\newtheorem*{rep@theorem}{\rep@title}
\newcommand{\newreptheorem}[2]{%
\newenvironment{rep#1}[1]{%
 \def\rep@title{#2 \ref{##1}}%
 \begin{rep@theorem}}%
 {\end{rep@theorem}}}
\makeatother

\newtheorem {theorem}{Theorem}
\newreptheorem{theorem}{Theorem}
\newtheorem {lemma}[theorem]{Lemma}
\newtheorem {proposition}[theorem]{Proposition}

\numberwithin{equation}{section}
\numberwithin{theorem}{section}

\theoremstyle{definition}

\newtheorem{remark}[theorem]{Remark}
\newtheorem*{remark*}{Remark}

\newlist{pcases}{enumerate}{1}
\setlist[pcases]{
  label=\bf{Case~\arabic*:}\protect\thiscase.~,
  ref=\arabic*,
  align=left,
  labelsep=0pt,
  leftmargin=0pt,
  labelwidth=0pt,
  parsep=0pt
}
\newcommand{\case}[1][]{%
  \if\relax\detokenize{#1}\relax
    \def\thiscase{}%
  \else
    \def\thiscase{~#1}%
  \fi
  \item
}

\newcommand{\Z}{\mathbb{Z}}

\newcommand{\F}{\mathbb{F}}
\newcommand{\Q}{\mathbb{Q}}

\newcommand\hfk{\mathit{HFK}}
\newcommand\hfkhat{\widehat{\hfk}}

\newcommand\KHI{\mathit{KHI}}

\DeclareFontFamily{U}{mathx}{\hyphenchar\font45}
\DeclareFontShape{U}{mathx}{m}{n}{
      <5> <6> <7> <8> <9> <10>
      <10.95> <12> <14.4> <17.28> <20.74> <24.88>
      mathx10
      }{}
\DeclareSymbolFont{mathx}{U}{mathx}{m}{n}
\DeclareFontSubstitution{U}{mathx}{m}{n}
\DeclareMathAccent{\widecheck}{0}{mathx}{"71}

\newcommand{\hfhat}{\widehat{\mathit{HF}}}

\newcommand{\mirror}[1]{\overline{#1}}

\newcommand{\Kh}{\mathit{Kh}}

\newcommand{\Khr}{\overline{\Kh}}

\newcommand{\dcover}{\Sigma_2}

\usetikzlibrary{calc,intersections}
\tikzset{every picture/.style=thick}
\tikzset{link/.style = { white, double = black, line width = 1.75pt, double distance = 1.25pt, looseness=1.75 }}
\tikzset{crossing/.style = {draw, circle, dotted, minimum size=0.5cm, inner sep=0, outer sep=0}}
\pgfplotsset{compat=1.12}

\begin{document}

\begin{abstract}
We give new proofs that Khovanov homology detects the figure eight knot and the cinquefoils, and that HOMFLY homology detects $5_2$ and each of the $P(-3,3,2n+1)$ pretzel knots.  For all but the figure eight these mostly follow the same lines as in  previous work. The key difference is that in honor of Tom Mrowka's 60th birthday, the arguments here use instanton Floer homology rather than knot Floer homology.
\end{abstract}

\maketitle
\section{Introduction}
\label{sec:intro}

Khovanov homology has been proved to detect a handful of the simplest knots, including the unknot \cite{km-unknot} and the trefoils $T(\pm2,3)$ \cite{bs-trefoil}.  Recently, Dowlin \cite{dowlin} constructed a spectral sequence from Khovanov homology to knot Floer homology, which made it possible to prove that Khovanov homology also detects the figure eight \cite{bdlls}, the cinquefoils $T(\pm2,5)$ \cite{bhs-cinquefoil}, and $5_2$ \cite{bs-nearly-fibered}.  In \cite{bs-nearly-fibered} we also used it to prove that reduced HOMFLY homology detects each of the $P(-3,3,2n+1)$ pretzel knots.

In this note, we give alternative arguments for most of these detection results, replacing knot Floer homology and Dowlin's spectral sequence with instanton knot homology and Kronheimer and Mrowka's spectral sequence from \cite{km-unknot}.  The figure eight detection result follows quickly from known facts about instanton L-space knots, including a criterion for their detection due to Li and Liang \cite{li-liang}.  For $T(\pm2,5)$, we apply this criterion together with the classification of genus-2 instanton L-space knots given in \cite[Corollary~1.8]{frw-cinquefoil}.  Our main results are the following, which do not make any use of Heegaard Floer homology.

\begin{theorem} \label{thm:main-41}
Let $K$ be a knot whose reduced Khovanov homology $\Khr(K)$ over some field is five-dimensional and supported in $\delta$-grading zero.  Then $K$ is the figure eight knot.
\end{theorem}

\begin{theorem} \label{thm:main-T25}
Let $K$ be a knot whose reduced Khovanov homology $\Khr(K)$ over some field is five-dimensional and supported in the single $\delta$-grading $\pm2$.  Then $K = T(\pm2,5)$.
\end{theorem}

Here the $\delta$-grading on reduced Khovanov homology is defined by $\delta = q/2-h$, where $q$ and $h$ are the quantum and homological gradings, respectively. 

\begin{theorem} \label{thm:main-52}
Let $K$ be a knot whose reduced HOMFLY homology $\bar{H}(K;\Q)$ is isomorphic to $\bar{H}(5_2;\Q)$ as triply-graded vector spaces.  Then $K = 5_2$.
\end{theorem}

\begin{theorem} \label{thm:main-pretzel}
Let $K$ be a knot, and suppose for some $n\in\Z$ that $\bar{H}(K;\Q)$ is isomorphic to $\bar{H}(P(-3,3,2n+1);\Q)$ as triply-graded vector spaces.  Then $K = P(-3,3,2n+1)$.
\end{theorem}

\begin{remark}
We should emphasize that the proofs of Theorem~\ref{thm:main-T25} and \ref{thm:main-pretzel} are not really new; they merely replace the parts of the arguments in \cite{bhs-cinquefoil,bs-nearly-fibered} which involve knot Floer homology.  By contrast, the proof of Theorem~\ref{thm:main-41} is genuinely different from the one in \cite{bdlls}.  We also point out that we do not know how to prove that Khovanov homology detects $5_2$, as in \cite{bs-nearly-fibered}, without appealing to Heegaard Floer homology.
\end{remark}

\begin{remark}
Beliakova, Putyra, Robert, and Wagner \cite{bprw} recently established a new spectral sequence relating HOMFLY homology with knot Floer homology. One can use their construction in place of Dowlin's spectral sequence, together with our results in \cite{bs-nearly-fibered}, to give additional alternative proofs of Theorems~\ref{thm:main-52} and \ref{thm:main-pretzel}.
\end{remark}

Building on Theorems~\ref{thm:main-41} and \ref{thm:main-T25}, we note that a knot $K$ with $\dim \Khr(K)=5$ satisfies $\det(K)=5$ if and only if $\Khr(K)$ is supported entirely in $\delta$-gradings of a single parity (see Lemma~\ref{lem:det-k}).  We do not claim here to classify all knots $K$ for which $\Khr(K)$ is 5-dimensional and $\det(K)=5$, but we do come close by allowing some Heegaard Floer input.

\begin{theorem} \label{thm:main-other}
Let $K$ be a knot other than the figure eight or $T(\pm2,5)$, and suppose that $\dim \Khr(K;\Z/2\Z) = 5$ and $\det(K)=5$.  Then $K$ is hyperbolic, with Seifert genus 4 and Alexander polynomial
\[ \Delta_K(t) = t^4 - t^3 + 1 - t^{-3} + t^{-4}. \]
Moreover, either $K$ or its mirror is an instanton L-space knot, hence fibered and strongly quasipositive, with signature $\pm8$.
\end{theorem}

\begin{remark}
By \cite[Theorem~1]{bdlls}, a genus-4 knot satisfying the hypotheses of Theorem~\ref{thm:main-other} cannot be supported in a single $\delta$-grading.
\end{remark}

\begin{remark}
The Alexander polynomial appearing in Theorem~\ref{thm:main-other} is in fact the Alexander polynomial of an instanton L-space knot, namely the $(5,2)$-cable of the right-handed trefoil.  But this knot $K = 13n_{4639}$ cannot satisfy the hypotheses of Theorem~\ref{thm:main-other}, since it is a satellite.  Indeed, according to \cite[\S7.2]{rasmussen-polynomials}, we have $\dim \Khr(K;\Q) = 11$, and $\dim \Khr(K;\Z/2\Z)$ is even larger, since $\Khr(K;\Q)$ has nontrivial 2-torsion.
\end{remark}

The organization is as follows.  In Section~\ref{sec:proofs} we prove Theorems~\ref{thm:main-41} and \ref{thm:main-T25}.  In Section~\ref{sec:higher-genus} we prove Theorem~\ref{thm:main-other} by first reducing it to a question about the factorization of the Alexander polynomial $\Delta_K(t)$, and then determining enough about the factorization to rule out all cases except $g(K)=4$.  Then in Section~\ref{sec:nearly-fibered} we prove Theorems~\ref{thm:main-52} and \ref{thm:main-pretzel}, about the HOMFLY homology of $5_2$ and the pretzels $P(-3,3,2n+1)$.

\subsection*{Acknowledgements}

We thank the referee for helpful feedback on the original version of this paper, and for suggesting an analytic proof of Proposition~\ref{prop:cyclotomic-product-even}.

\section{The figure eight and the cinquefoils}
\label{sec:proofs}

We begin with the following lemmas, which are certainly well known to experts.  In what follows we let $s(K) \in 2\Z$ denote the Rasmussen $s$-invariant \cite{rasmussen-s}.

\begin{lemma} \label{lem:det-k}
The determinant of a knot $K$ is determined by its $\delta$-graded reduced Khovanov homology $\Khr(K;\Q)$.  In particular, we have
\[ \dim \Khr(K;\Q) \geq \det(K), \]
and equality holds if and only if $\Khr(K;\Q)$ is supported entirely in even $\delta$-gradings or entirely in odd $\delta$-gradings.  If in fact it is supported in a single $\delta$-grading $\sigma$, then $s(K) = 2\sigma$.
\end{lemma}

\begin{proof}
The claim about $s(K)$ follows from the fact that $\Khr(K;\Q)$ is supported in one $\delta$-grading if and only if the unreduced Khovanov homology $\Kh(K;\Q)$ is ``H-thin'' \cite[Proposition~3.6]{khovanov-patterns}, meaning supported in two $\delta$-gradings, in which case these $\delta$-gradings are necessarily $\sigma\pm\frac{1}{2}$.  Then the part of $\Kh(K;\Q)$ in homological grading $h=0$ must be supported in quantum gradings $q = 2(h+\delta) = 2\sigma\pm1$, and so $s(K) = 2\sigma$ by definition.

For the claim about $\det(K)$, it follows from the fact that $\Khr(K)$ categorifies the Jones polynomial \cite{khovanov-patterns}:
\[ V_K(t) = \sum_{q,h\in\Z} (-1)^h t^{q/2} \dim \Khr^{h,q}(K). \]
Then $\det(K)$ is equal to the absolute value of
\begin{align*}
V_K(-1) &= \sum_{q,h\in\Z} (-1)^h (-1)^{q/2} \dim \Khr^{h,q}(K) \\
&= \sum_{q/2-h\text{ even}} \dim \Khr^{h,q}(K) - \sum_{q/2-h\text{ odd}} \dim \Khr^{h,q}(K) \\
&= \dim \Khr^{\delta\equiv0\!\!\!\!\!\pmod{2}}(K) - \dim \Khr^{\delta\equiv1\!\!\!\!\!\pmod{2}}(K).
\end{align*}
Thus $\dim \Khr(K;\Q) \geq \det(K)$ follows immediately from the triangle inequality.  Moreover, the last expression above is equal to $\pm \dim \Khr(K)$ if and only if one of the two terms is zero, or equivalently if and only if $\Khr(K)$ is supported entirely in even $\delta$-gradings or entirely in odd $\delta$-gradings, as claimed.
\end{proof}

\begin{lemma} \label{lem:f2-to-q}
If $\Khr(K;\F)$ is 5-dimensional for some field $\F$, then so is $\Khr(K;\Q)$, and moreover the two are supported in the same $\delta$-gradings.
\end{lemma}

\begin{proof}
This follows from two applications of the universal coefficient theorem.  First, we use it to show that either $\Khr(K;\Q)$ or $\Khr(K;\Z/p\Z)$ is also $5$-dimensional, depending on whether $\F$ has characteristic $0$ or $p>0$ respectively.  Then in the case $\F = \Z/p\Z$, the lemma follows from the universal coefficient theorem, as applied to $\Khr(K;\Z)$, and the fact that $\dim_\Q \Khr(K;\Q)$ is odd and neither 1 nor 3, exactly as at the start of \cite[\S5]{bhs-cinquefoil}.
\end{proof}

We next attempt to use $\Khr(K;\Q)$ to determine the instanton knot homology $\KHI(K;\Q)$, as defined by Kronheimer and Mrowka \cite{km-excision}.  This invariant comes equipped with an Alexander grading
\[ \KHI(K) \cong \bigoplus_{i=-g}^g \KHI(K,i), \]
where $g=g(K)$ is the genus of $K$, each of whose summands are $\Z/2\Z$-graded.  This decomposition recovers the Alexander polynomial of $K$ by the relation
\begin{equation} \label{eq:categorify}
\pm\Delta_K(t) = \sum_{i=-g}^g \chi(\KHI(K,i)) \cdot t^i,
\end{equation}
as proved in \cite{km-alexander,lim}.  (The sign on the left comes from differing conventions for the $\Z/2\Z$ grading.). It also satisfies the following properties:
\begin{itemize}
\item Symmetry: $\KHI(K,i) \cong \KHI(K,-i)$ as $\Z/2\Z$-graded vector spaces for all $i$.
\item Genus detection: $\dim \KHI(K,g) \geq 1$.
\item Fiberedness detection: $\dim \KHI(K,g) = 1$ if and only if $K$ is fibered.
\end{itemize}
The symmetry follows from the remark after \cite[Proposition~7.1]{km-excision}, and the genus detection and fiberedness results are \cite[Proposition~7.16]{km-excision} and \cite[Proposition~4.1]{km-alexander} respectively.

\begin{proposition} \label{prop:dim-khi}
Let $K$ be a knot for which $\dim \Khr(K;\Q) = 5$.  Then $\KHI(K;\Q)$ has total dimension 5.
\end{proposition}

\begin{proof}
Kronheimer and Mrowka \cite[Proposition~1.2]{km-unknot} constructed a spectral sequence
\[ \Khr(K;\Q) \ \Longrightarrow\ I^\natural(\mirror{K};\Q) \]
which converges to the singular instanton knot homology of the mirror of $K$.  Since the rank of the latter is invariant under mirroring, this yields a rank inequality
\[ \dim I^\natural(K;\Q) \leq \dim \Khr(K;\Q) = 5. \]
Moreover, we know by \cite[Proposition~1.4]{km-unknot} that $I^\natural(K;\Q) \cong \KHI(K;\Q)$.  Thus
\begin{equation} \label{eq:khi-dim-bound}
\dim \KHI(K;\Q) \leq 5.
\end{equation}

Equation~\eqref{eq:categorify} and the fact that $\Delta_K(1) = 1$ tell us that the total rank of $\KHI(K)$ must be odd.  It cannot be 1, because then $K$ would be the unknot -- otherwise the summands $\KHI(K,g)$ and $\KHI(K,-g)$ are distinct and contribute at least 1 each to $\dim \KHI(K)$ -- and it cannot be 3 or else $K$ would be a trefoil \cite[Theorem~1.6]{bs-trefoil}.  Thus we use \eqref{eq:khi-dim-bound} to conclude that $\dim \KHI(K) = 5$.
\end{proof}

\begin{proposition} \label{prop:41-or-lspace}
Let $K$ be a knot for which $\dim \KHI(K;\Q) = 5$.  Then exactly one of the following is true:
\begin{itemize}
\item $K$ has Alexander polynomial 1, and in particular $\det(K)=1$.
\item $K$ is the figure eight knot.
\item $K$ has genus $g \geq 2$ and instanton knot homology
\[ \KHI(K;\Q) \cong \Q_g \oplus \Q_{g-1} \oplus \Q_0 \oplus \Q_{1-g} \oplus \Q_{-g}, \]
where the subscripts denote the Alexander grading of each summand.  In this case either $K$ or its mirror is an instanton L-space knot.
\end{itemize}
In particular, if $\det(K) \neq 1$ then $K$ is fibered.
\end{proposition}

\begin{proof}
Supposing first that $K$ is not fibered, then
\[ \dim \KHI(K,g) + \dim \KHI(K,-g) \leq 5, \]
and both terms on the left are equal and greater than 1 by the symmetry and fiberedness detection properties, so
\[ \dim \KHI(K,g) = \dim \KHI(K,-g) = 2. \]
By symmetry the remaining $\Q$ summand of $\KHI(K)$ can only be in Alexander grading zero, so
\[ \KHI(K) \cong \Q^2_g \oplus \Q^{\vphantom{2}}_0 \oplus \Q^2_{-g}. \]

We apply \eqref{eq:categorify} to determine $\Delta_K(t)$.  If $\KHI(K,g)$ is supported in a single $\Z/2\Z$ grading, then we have
\[ \Delta_K(t) = \pm 2(t^g + t^{-g}) \pm 1 \]
for some signs, and this is impossible because there is no choice of signs for which $\Delta_K(1) = 1$.  So $\KHI(K,g)$ must have a copy of $\Q$ in each $\Z/2\Z$ grading, which by \eqref{eq:categorify} tells us that $\Delta_K(t) = 1$.  Then $\det(K) = |\Delta_K(-1)| = 1$.

Next, supposing that $K$ is fibered of genus $1$, then $K$ must be a trefoil or the figure eight.  But if $K$ is a trefoil then $\dim \KHI(K) = 3$, so in fact $K$ can only be the figure eight.

Finally, if $K$ is fibered of genus $g \geq 2$, then we proved in \cite[Theorem~1.7]{bs-trefoil} that $\dim \KHI(K,g-1) \geq 1$, so $\KHI(K)$ is 5-dimensional and is nonzero at least in the four distinct Alexander gradings $\pm g$ and $\pm(g-1)$.  Again by symmetry each of these summands must be 1-dimensional and the remaining $\Q$ summand must be in degree 0, so
\[ \KHI(K) \cong \Q_g \oplus \Q_{g-1} \oplus \Q_0 \oplus \Q_{1-g} \oplus \Q_{-g} \]
as claimed.  Now Li and Liang \cite[Theorem~1.4]{li-liang} proved that any knot for which $\KHI$ has this form must be an instanton L-space knot, up to mirroring, so this completes the proof.
\end{proof}

\begin{proposition} \label{prop:s-genus}
Let $K$ be a knot satisfying $\dim \Khr(K;\Q) = 5$ and $\det(K) \neq 1$.  Then
\[ s(K) = \begin{cases}
\hphantom{-2g(K)}\llap{$0$} & \text{if } K \text{ is the figure eight knot} \\
\hphantom{-}2g(K) & \text{if } K \text{ is an instanton L-space knot} \\
-2g(K) & \text{if } \mirror{K} \text{ is an instanton L-space knot},
\end{cases} \]
and exactly one of these cases occurs.  We also have $g(K) \geq 2$ unless $K$ is the figure eight.
\end{proposition}

\begin{proof}
The case where $K$ is the figure eight is immediate, so we will suppose that $K$ is some other knot.  Proposition~\ref{prop:dim-khi} says that $\dim \KHI(K) = 5$, so we can apply Proposition~\ref{prop:41-or-lspace} to see that either $K$ or its mirror is an instanton L-space knot, and that $g(K) \geq 2$.

Let us suppose that $K$ is an instanton L-space knot, rather than $\mirror{K}$.  Then $K$ is strongly quasipositive \cite[Theorem~1.15]{bs-lspace}, so we know from \cite[Proposition~4]{plamenevskaya-transverse} or \cite[Proposition~1.7]{shumakovitch-slice-bennequin} that $s(K) = 2g(K)$.  Otherwise $\mirror{K}$ is an instanton L-space knot, so the same argument says that
\[ s(K) = -s(\mirror{K}) = -2g(\mirror{K}) = -2g(K). \qedhere \]
\end{proof}

We are now ready to prove Theorems~\ref{thm:main-41} and \ref{thm:main-T25}.

\begin{proof}[Proof of Theorem~\ref{thm:main-41}]
Suppose that $\Khr(K;\F)$ is 5-dimensional and supported in the single $\delta$-grading $\sigma=0$.  Then the same is true of $\Khr(K;\Q)$ by Lemma~\ref{lem:f2-to-q}.  Lemma~\ref{lem:det-k} now tells us that $\det(K)=5$ and that $s(K) = 2\sigma = 0$, so $K$ must be the figure eight by Proposition~\ref{prop:s-genus}.
\end{proof}

\begin{proof}[Proof of Theorem~\ref{thm:main-T25}]
Suppose that $\Khr(K;\F)$ is 5-dimensional and supported in the single $\delta$-grading $\sigma=2$.  Then the same is true of $\Khr(K;\Q)$ by Lemma~\ref{lem:f2-to-q}, so Lemma~\ref{lem:det-k} says that $\det(K)=5$ and that $s(K)=2\sigma=4$.  Thus Proposition~\ref{prop:s-genus} says that $K$ must be an instanton L-space knot of genus 2.  Such knots are fibered, and in \cite[\S2]{blsy} we gave a partial characterization of their possible monodromies; more recently, Farber, Reinoso, and Wang \cite[Corollary~1.8]{frw-cinquefoil} used this to show that $K$ is necessarily $T(2,5)$. We remark that if we specifically wanted to work over $\Z/2\Z$, then we could finish the proof that $K=T(2,5)$ without recourse to \cite{frw-cinquefoil}, using instead the arguments in \cite{blsy}.

Now if $\Khr(K;\F)$ is 5-dimensional and supported in the $\delta$-grading $\sigma=-2$, then we apply the above to its mirror $\mirror{K}$ to conclude that $\mirror{K} = T(2,5)$, and hence that $K = T(-2,5)$.  This completes the proof.
\end{proof}

\section{Other determinant-5 knots}
\label{sec:higher-genus}

Here we address the question of whether there are knots $K$ other than the figure eight and cinquefoils such that $\dim \Khr(K;\Z/2\Z) = 5$ and $\Khr(K;\Z/2\Z)$ is supported entirely in $\delta$-gradings of a single parity.

In this case, Lemma~\ref{lem:f2-to-q} and Proposition~\ref{prop:s-genus} tell us that if $K$ is not the figure eight, then either $K$ or its mirror is an instanton L-space knot whose genus $g(K) = \frac{1}{2}s(K)$ is at least $2$.  We will see that its branched double cover is a Heegaard Floer L-space, and then use recent work of Boileau, Boyer, and Gordon \cite{boileau-boyer-gordon-1} to put strong conditions on the Alexander polynomial of $K$ which rule out all cases except $g(K)=4$.

\subsection{The Alexander polynomial of a thin knot}

\begin{proposition} \label{prop:large-genus-alexander}
Suppose that $\Khr(K;\Q)$ is 5-dimensional and that $\det(K) = 5$, but that $K$ is not the figure eight or $T(\pm2,5)$.  Then either $K$ or its mirror is an instanton L-space knot, hence fibered and strongly quasipositive, and its Alexander polynomial is
\[ \Delta_K(t) = t^g - t^{g-1} + 1 - t^{1-g} + t^{-g} \]
where $g=g(K)$ is even and at least $4$.
\end{proposition}

\begin{proof}
Proposition~\ref{prop:dim-khi} tells us once again that $\dim \KHI(K) = 5$, so that either $K$ or its mirror is an instanton L-space knot (hence fibered and strongly quasipositive \cite[Theorem~1.15]{bs-lspace}) and
\[ \KHI(K;\Q) \cong \Q_g \oplus \Q_{g-1} \oplus \Q_0 \oplus \Q_{1-g} \oplus \Q_{-g} \]
by Proposition~\ref{prop:41-or-lspace}.  This determines the Alexander polynomial $\Delta_K(t)$ by \eqref{eq:categorify}: applying the conditions $\Delta_K(1)=1$ and $\Delta_K(-1) = \pm \det(K) = \pm 5$, we must have
\[ \Delta_K(t) = (-1)^g \left( t^g - t^{g-1} + (-1)^g - t^{1-g} + t^{-g}\right). \]
In fact, since $K$ is an instanton L-space knot up to mirroring, we can apply \cite[Theorem~1.9]{li-ye-surgery} to deduce that the nonzero coefficients of $\Delta_K(t)$ alternate in sign, so $g$ must be even and then
\[ \Delta_K(t) = t^g - t^{g-1} + 1 - t^{1-g} + t^{-g}. \]
By assumption $g$ is not $T(\pm2,5)$, but there are no other instanton L-space knots of genus $2$ \cite[Corollary~1.8]{frw-cinquefoil}, so then $g$ is at least $4$.
\end{proof}

\begin{proposition} \label{prop:bdc-argument}
Suppose that $\Khr(K;\Z/2\Z)$ is 5-dimensional and that $\det(K)=5$, but that $K$ is not the figure eight or $T(\pm2,5)$.  Then $K$ has signature $\sigma(K) = \pm2g(K)$, and if we write $h = \frac{1}{2}g(K)$, then $h$ is an integer with $h \geq 2$,
and the polynomial
\[ p_h(t) = t^{4h} - t^{4h-1} + t^{2h} - t + 1 \]
is a product of cyclotomic polynomials.
\end{proposition}

\begin{proof}
Lemma~\ref{lem:f2-to-q} says that $\Khr(K;\Q)$ is also 5-dimensional, so Proposition~\ref{prop:large-genus-alexander} tells us that $g(K)$ is even and at least $4$, so that $h\in\Z$ as claimed; that $K$ is fibered and strongly quasipositive, after possibly replacing it with its mirror; and that
\[ p_h(t) = t^{g(K)} \Delta_K(t), \]
so that $p_h(t)$ is a product of cyclotomic polynomials if and only if $\Delta_K(t)$ is.  Boileau, Boyer, and Gordon \cite[Corollary~1.2]{boileau-boyer-gordon-1} proved that if the branched double cover of a fibered, strongly quasipositive knot $K$ is a Heegaard Floer L-space, then $\Delta_K(t)$ is a product of cyclotomic polynomials; this follows from their observation \cite[Proposition~6.1]{boileau-boyer-gordon-1} that in this case $K$ has signature $\pm2g(K)$.  Thus it suffices to show that $\dcover(K)$ is a Heegaard Floer L-space.

We now apply Ozsv\'ath and Szab\'o's link surgeries spectral sequence \cite{osz-branched}, and in particular the inequality
\[ \det(K) \leq \dim \hfhat(\dcover(K);\Z/2\Z) \leq \dim \Khr(K;\Z/2\Z), \]
to conclude that $\dim \hfhat(\dcover(K);\Z/2\Z) = |H_1(\dcover(K))| = 5$.  In other words, the branched double cover $\dcover(K)$ is a Heegaard Floer L-space, and the proposition follows.
\end{proof}

\begin{proof}[Proof of Theorem~\ref{thm:main-other}]
Suppose that $\Khr(K;\Z/2\Z)$ is 5-dimensional and supported in $\delta$-gradings of a single parity.  Then the same is true of $\Khr(K;\Q)$ by Lemma~\ref{lem:f2-to-q}, and $\det(K)=5$ by Lemma~\ref{lem:det-k}.  Moreover, since $K$ is not the figure eight by assumption, Proposition~\ref{prop:s-genus} says that $s(K) = \pm2g(K)$ and $g(K) \geq 2$.

We apply Proposition~\ref{prop:bdc-argument} and see that in fact $g(K)$ is even and at least $4$, that $\sigma(K) = \pm2g(K)$, and that the polynomial $p_h(t)$ must be a product of cyclotomic polynomials, where $h = \frac{1}{2}g(K)$.  In Proposition~\ref{prop:cyclotomic-product-even} we will prove that this is not the case for any $h \geq 3$, so we must have $h \leq 2$.  But since $2h = g(K) \geq 4$ this leaves only $h=2$, hence $g(K)=4$ as claimed.  Proposition~\ref{prop:large-genus-alexander} establishes all of the remaining conclusions except hyperbolicity.  The Alexander polynomial and genus prevent $K$ from being a torus knot, so we need only show that it cannot be a satellite; this requires substantially different techniques, so we defer it to Proposition~\ref{prop:winding-number-1} below.
\end{proof}

We need the following lemma to prove that knots satisfying the hypotheses of Theorem~\ref{thm:main-other} cannot be satellites.

\begin{lemma} \label{lem:bdc-pattern}
Let $P \subset S^1\times D^2$ be a knot, and suppose that there is a branched double cover
\[ S^1 \times D^2 \to S^1\times D^2 \]
with branch locus $P$.  Then $P$ is isotopic to the core $S^1 \times \{0\}$.
\end{lemma}

\begin{proof}
Fix a nontrivial torus knot $T = T(p,q)$, and consider the satellite $K = P(T)$.  We form the branched double cover $\dcover(K)$ by taking a double cover $X$ of $S^3 \setminus N(T)$, which may or may not be connected a priori, and gluing it to the branched double cover $\Sigma_P$ of $P \subset S^1\times D^2$.  Since the latter has connected boundary, so does $X$, so $X$ is connected.  We also know that $X$ is Seifert fibered over a disk, since the same is true of $S^3 \setminus N(T)$; and that it has $r=2$ singular fibers since it is not a solid torus.  (We recall that the knot complement has base orbifold $D^2(|p|,|q|)$.)

Gluing the solid torus $\Sigma_P$ to $X$ amounts to a Dehn filling of $X$, so by \cite[Proposition~2]{heil} there are now two possibilities:
\begin{itemize}
\item the Seifert fibration extends to $\dcover(K)$, with at most $r+1=3$ singular fibers, or
\item we have filled the fiber slope, and the resulting $\dcover(K)$ is a connected sum of $r$ nontrivial lens spaces.
\end{itemize}
In the latter case, $K$ is a connected sum of $r$ nontrivial knots, because the branched double cover of a prime knot is prime  \cite[Proposition~5.1]{hedden-ni}.  But as a satellite of a torus knot, it can only be composite if the pattern $P$ has wrapping number $1$, see e.g.\ \cite[Theorem~4.4.1]{cromwell-book}.  In this case $P$ is a connected sum of the core $C_0 = S^1 \times \{0\}$ with some other knot $K_0 \subset S^3$, and we have
\[ \Sigma_P \cong \dcover(C_0) \# \dcover(K_0) \cong (S^1\times D^2) \# \dcover(K_0). \]
Thus $\dcover(K_0) \cong S^3$, which implies that $K_0$ is unknotted \cite{waldhausen-involution} and hence that $P$ is isotopic to the core $C_0$.

In the remaining case, we know that $\dcover(K)$ is a small Seifert fibered space: it has base $S^2$ and at most three singular fibers.  We can also arrange for $\pi_1(\dcover(K))$ to be infinite by taking both $p$ and $q$ to be large.  Then a folklore result (see e.g.\ \cite[Proposition~3.3]{motegi-note} for details) says that $K$ must be either a torus knot or a Montesinos knot with three rational tangles.  These are never nontrivially satellite knots -- the Montesinos case is due to Oertel \cite[Corollary~4]{oertel-star} -- so either $P$ is contained in a ball inside $S^1\times D^2$, in which case its branched double cover cannot actually be $S^1\times D^2$, or $P$ is isotopic to the core circle $S^1 \times \{0\}$, as claimed.
\end{proof}

The following is the last remaining claim of Theorem~\ref{thm:main-other}.

\begin{proposition} \label{prop:winding-number-1}
Let $K$ be a knot such that $\dim \Khr(K;\Z/2\Z) = 5$ and $\det(K) = 5$.  Then $K$ is not a satellite knot.
\end{proposition}

\begin{proof}
We first show that $K$ is prime: if $K \cong K_1 \# K_2$, then by the K\"unneth formula for reduced Khovanov homology over $\Z/2\Z$ we have
\[ \dim \Khr(K_1) \cdot \dim \Khr(K_2) = \dim \Khr(K) = 5, \]
so $\dim \Khr(K_i) = 1$ for some $i$, but then $K_i$ must be the unknot \cite{km-unknot}.  This also implies in turn that the branched double cover $\dcover(K)$ is prime, as in \cite[Proposition~5.1]{hedden-ni}.

Now we suppose that $K$ is a nontrivial satellite, with pattern $P \subset S^1 \times D^2$ and companion $C \subset S^3$.  By assumption $P$ cannot be isotopic to the core $S^1 \times \{0\}$, and $C$ cannot be the unknot.  We recall from Theorem~\ref{thm:main-other} that $K$ (up to mirroring) is fibered and strongly quasipositive, and from the proof of Proposition~\ref{prop:bdc-argument} that $\dcover(K)$ is a Heegaard Floer L-space.  It thus follows from \cite[Proposition~6.2 and Remark~6.3]{boileau-boyer-gordon-1} that $P$ must have winding number 1.

Since $P$ has odd winding number, we can write $\dcover(K)$ as a union
\[ \dcover(K) \cong X_2(C) \cup_{T^2} \Sigma_P, \]
where $X_2(C)$ is a connected double cover of the exterior $S^3 \setminus N(C)$, and $\Sigma_P$ is a double cover of $S^1\times D^2$ branched over $P$; these pieces are glued along their respective torus boundaries.  We note that $\partial X_2(C)$ is incompressible since $C$ is a nontrivial knot.

Suppose that $\Sigma_P$ has incompressible boundary as well.  Then this torus remains incompressible in $\dcover(K)$.  Hanselman, Rasmussen, and Watson \cite[Theorem~7.20]{hrw} classified the prime, toroidal Heegaard Floer L-spaces $Y$ with $|H_1(Y)|=5$, and showed in particular that they are all built by gluing together a pair of trefoil exteriors.  In each case there is a unique incompressible torus up to isotopy, so we conclude that $X_2(C)$ and $\Sigma_P$ are both trefoil exteriors.  In particular we have a double cover
\[ S^3 \setminus N(T(\pm2,3)) \cong X_2(C) \to S^3 \setminus N(C). \]
Gonzalez-Acu\~na and Whitten \cite[Theorem~3.4]{gonzalez-acuna-whitten} proved in this case that either
\begin{itemize}
\item $C$ is not a torus knot, and then it must admit a cyclic $\pm2$-surgery, which contradicts the main result of \cite{km-su2}; or
\item $C$ is a torus knot $T(p,q)$, and we can write $2 = dpq\pm1$ for some integer $d$, which is also impossible.
\end{itemize}
Thus $\Sigma_P$ has compressible boundary after all, and we can write
\[ \Sigma_P \cong (S^1\times D^2) \# Z \]
for some closed 3-manifold $Z$, which may or may not be $S^3$.

Supposing that $Z$ is different from $S^3$, it now follows that
\[ \dcover(P(C)) \cong \big( X_2(C) \cup (S^1\times D^2) \big) \# Z \]
can only be prime if $S^3$ arises as a Dehn filling of $X_2(C)$, i.e., if $X_2(C)$ is the exterior of some other knot in $S^3$.  Again this is impossible since $C$ is nontrivial \cite[Theorem~3.4]{gonzalez-acuna-whitten}, so $\dcover(K) \cong \dcover(P(C))$ is not prime, which is a contradiction.  So $Z \cong S^3$, and therefore $\Sigma_P$ is a solid torus.  Lemma~\ref{lem:bdc-pattern} now tells us that $P$ must be isotopic to a core of $S^1 \times D^2$.  But in this case $K$ is not a nontrivial satellite of $C$ after all, so we are done.
\end{proof}

\subsection{Factorization of the Alexander polynomial}

In this subsection we prove the following, which completes the proof of Theorem~\ref{thm:main-other}.

\begin{proposition} \label{prop:cyclotomic-product-even}
Fix an integer $h \geq 1$, and define the polynomial
\begin{equation} \label{eq:p_h}
p_h(t) = t^{4h} - t^{4h-1} + t^{2h} - t + 1.
\end{equation}
If $h \geq 3$, then $p_h(t)$ is not a product of cyclotomic polynomials.
\end{proposition}

\begin{remark} \label{eq:factor-small-even-g}
By contrast, we note that $p_1(t) = \Phi_{10}(t)$ and $p_2(t) = \Phi_{10}(t)\cdot \Phi_{12}(t)$.
\end{remark}

In the proof of Proposition~\ref{prop:cyclotomic-product-even}, we will adapt an algorithm called the ``Graeffe'' method \cite{bradford-davenport} for recognizing cyclotomic polynomials.  The idea is that if
\[ p(t) = t^d + a_{d-1}t^{d-1} + \dots + a_1t + a_0 \]
has roots $\alpha_1,\alpha_2,\dots,\alpha_d$, then we can split $p$ into its even and odd parts by writing
\[ p(t) = p_e(t^2) + t\cdot p_o(t^2) \quad\text{where}\quad
\begin{cases}
p_e(t) = a_0 + a_2t + a_4t^2 + \dots, &\\
p_o(t) = a_1 + a_3t + a_5t^2 + \dots. &
\end{cases} \]
Then Graeffe's root-squaring method says that the polynomial
\[ q(t) = (-1)^d\left( p_e(t)^2 - t\cdot p_o(t)^2 \right) \]
has roots $\alpha_1^2, \alpha_2^2, \dots, \alpha_d^2$.  For example, if $p(t) = \Phi_n(t)$ then we will have
\begin{equation} \label{eq:cyclotomic-q}
q(t) = \begin{cases} \Phi_n(t), & n\text{ odd} \\ \Phi_{n/2}(t), & n\text{ even but not a multiple of }4 \\ \left(\Phi_{n/2}(t)\right)^2, & n\text{ a multiple of 4}. \end{cases}
\end{equation}

We note that if $p(t) = p_h(t)$ is the polynomial given in \eqref{eq:p_h}, then we have
\begin{align*}
p_e(t) &= t^{2h} + t^h + 1, &
p_o(t) &= -t^{2h-1} - 1
\end{align*}
and so the root-squaring method produces the polynomial
\begin{equation} \label{eq:q_h}
\begin{aligned}
q_h(t) &= (t^{2h} + t^h + 1)^2 - t(-t^{2h-1}-1)^2 \\
&= t^{4h} - t^{4h-1} + 2t^{3h} + t^{2h} + 2t^{h} - t + 1.
\end{aligned}
\end{equation}

Before we begin the proof of Proposition~\ref{prop:cyclotomic-product-even}, we will first recall some facts about special values of cyclotomic polynomials.

\begin{lemma} \label{lem:special-values}
For all $n \geq 2$, we have
\[
\Phi_n(1) = \begin{cases} p & n=p^e \text{ is a prime power} \\ 1 & otherwise \end{cases}
\quad\text{and}\quad
\Phi_n(-1) = \begin{cases}
1 & n\text{ odd} \\
\Phi_{n/2}(1) & n\text{ even}.
\end{cases}
\]
\end{lemma}

\begin{proof}
We evaluate both sides of
\[ x^{n-1} + x^{n-2} + \dots + x + 1 = \prod_{\substack{d\mid n \\ d \neq 1}} \Phi_d(x) \]
at $x=1$ and at $x=-1$ to conclude that 
\begin{align*}
\prod_{\substack{d\mid n \\ d \neq 1}} \Phi_d(1) &= n, &
\prod_{\substack{d\mid n \\ d \neq 1}} \Phi_d(-1) &= \begin{cases} 0 & n\text{\ even} \\ 1 & n\text{ odd}. \end{cases}
\end{align*}
If $n=p^e$ is a prime power, with $e \geq 1$, then the first of these implies by induction on $e$ that $\Phi_{p^e}(1) = p$, and then it follows that $\Phi_n(1) = 1$ if $n$ is not a prime power.  The second equation similarly implies by induction that $\Phi_n(-1) = 1$ for all odd $n$.

In the remaining cases, we wish to evaluate $\Phi_n(-1)$ where $n=2k$ is even.  We observe that if $k$ is odd then $\Phi_n(x) = \Phi_k(-x)$, and if $k$ is even then $\Phi_n(x) = \Phi_k(x^2)$.  In either case it follows that $\Phi_n(-1) = \Phi_k(1)$.
\end{proof}

We now begin to determine which cyclotomic polynmoials can divide $p_h(t)$.

\begin{lemma} \label{lem:cyclotomic-odd}
If $n$ is odd, then $\Phi_n(t)$ does not divide $p_h(t)$.
\end{lemma}

\begin{proof}
Suppose that $\Phi_n(t)$ divides $p_h(t)$.  Then the squares of the primitive $n$th roots of unity are also primitive $n$th roots of unity, so $\Phi_n(t)$ also divides the polynomial $q_h(t)$ from \eqref{eq:q_h}, and hence it divides the difference
\[ q_h(x) - p_x(h) = 2t^{3h} + 2t^{h} = 2t^{h}(t^{2h}+1). \]
The roots of $\Phi_n(t)$ are all nonzero, so it divides $t^{2h}+1$ and hence $t^{4h}-1$.  But this means that $n$ is an odd divisor of $4h$, so in fact $n$ divides $h$.

Letting $\zeta$ be any root of $\Phi_n(t)$, we have $\zeta^n=1$ and therefore $\zeta^h=1$.  We compute that
\begin{align*}
0 = p_h(\zeta) &= \zeta^{4h} - \zeta^{4h-1} + \zeta^{2h} - \zeta + 1 \\
&= 1 - \zeta^{-1} + 1 - \zeta + 1
\end{align*}
and so $\zeta^2 - 3\zeta + 1 = 0$.  But then $\zeta = \frac{1}{2}(3\pm\sqrt{5})$ is not a root of unity, contradiction.
\end{proof}

\begin{lemma} \label{lem:cyclotomic-2-mod-4}
Let $n=2k$ be twice an odd integer $k \geq 1$, and suppose that $\Phi_n(t)$ divides $p_h(t)$.  Then $n=10$, and $h$ is congruent to either $1$ or $2$ modulo 5.
\end{lemma}

\begin{proof}
Since $\Phi_n(t)$ divides $p_h(t)$, the primitive $k$th roots of unity must be roots of $q_h(t)$, so $\Phi_k(t)$ divides $q_h(t)$.  Equivalently, since $\Phi_n(t) = \Phi_k(-t)$ we see that $\Phi_n(t)$ divides $q_h(-t)$ and thus also the difference
\begin{align*}
q_h(-t) - p_h(t) &= 2t^{4h-1} + 2(-1)^ht^{3h} + 2(-1)^h t^h + 2t \\
&= 2t\left(t^{h-1} + (-1)^h\right)\left(t^{3h-1}+(-1)^h\right).
\end{align*}
Then $\Phi_n(t)$ divides either $t^{2h-2}-1$ or $t^{6h-2}-1$, and hence $n$ divides either $2h-2$ or $6h-2$.

Suppose first that $n$ divides $2h-2$, and let $\zeta$ be a root of $\Phi_n(t)$.  Then $\zeta^{2h-2} = 1$, so
\begin{align*}
0 = p_h(\zeta) &= \zeta^{4h} - \zeta^{4h-1} + \zeta^{2h} - \zeta + 1 \\
&= \zeta^4 - \zeta^3  + \zeta^2 - \zeta^1 + 1 = \Phi_{10}(\zeta)
\end{align*}
and since $\zeta$ is a root of the irreducible $\Phi_{10}(t)$, we must have $n=10$.  In this case $2h-2$ is a multiple of $10$, so $h\equiv 1\pmod{5}$.

Now suppose instead that $n$ divides $6h-2$, and let $\zeta$ be a root of $\Phi_n(t)$; then $\zeta^{6h-2}=1$.  We note that $n$ is not a multiple of $3$ since $6h-2$ is not, so $\zeta^3$ is also a primitive $n$th root of unity and therefore a root of $\Phi_n(t)$.  Since $\Phi_n(t)$ divides $p_h(t)$, we have
\begin{align*}
0 = p_h(\zeta^3) &= \zeta^{12h} - \zeta^{12h-3} + \zeta^{6h} - \zeta^3 + 1 \\
&= \zeta^4 - \zeta + \zeta^2 - \zeta^3 + 1 = \Phi_{10}(\zeta)
\end{align*}
and so once again we must have $n=10$.  Now $6h-2$ is a multiple of $10$, and so $h\equiv 2\pmod{5}$.
\end{proof}

\begin{lemma} \label{lem:10-multiples-of-4}
Fix $h \geq 1$.  If $p_h(t)$ is a product of cyclotomic polynomials, then we have
\begin{equation} \label{eq:10-4}
p_h(t) = \Phi_{10}(t) \cdot \prod_{j=1}^k \Phi_{n_j}(t)
\end{equation}
where each $n_j$ is a multiple of 4 but not a power of 2.  In particular $h$ must be congruent to either $1$ or $2$ modulo $5$.
\end{lemma}

\begin{proof}
By assumption we can find integers $n_0,n_1,\dots,n_k$ ($k\geq 0$) such that
\[ p_h(t) = \Phi_{n_0}(t) \cdot \Phi_{n_1}(t) \cdot \ldots \cdot \Phi_{n_k}(t). \]
By Lemma~\ref{lem:cyclotomic-odd}, all of the $n_j$ must be even, and then by Lemma~\ref{lem:cyclotomic-2-mod-4} they must all be either $10$ or multiples of $4$.  Setting $t=-1$, we have
\[ 5 = p_h(-1) = \prod_{j=0}^k \Phi_{n_j}(-1), \]
and every factor on the right is a nonnegative integer by Lemma~\ref{lem:special-values}; we order them so that $\Phi_{n_0}(-1)=5$ and $\Phi_{n_j}(-1)=1$ for all $j\geq 1$.

Now by Lemma~\ref{lem:special-values}, the integer $\Phi_{n_0}(-1)$ can only be $5$ if $n_0 = 2\cdot 5^e$ for some $e\geq 1$, and then $n_0$ is not a multiple of $4$ so it must be $10$.  Then Lemma~\ref{lem:cyclotomic-2-mod-4} guarantees that $h$ is either $1$ or $2$ modulo 5.  Moreover, we cannot have $n_j=10$ for any other $j \geq 1$, because then $\Phi_{n_j}(-1)$ would not be $1$, so the remaining $n_j$ are all multiples of $4$.

Finally, if instead we set $t=1$ then $p_h(1)=1$ implies that $\Phi_{n_j}(1)=1$ for all $j=1,2,\dots,k$, so we cannot have $n_j = 2^e$ because then Lemma~\ref{lem:special-values} would tell us that $\Phi_{n_j}(1)=2$ instead.
\end{proof}

With Lemma~\ref{lem:10-multiples-of-4} in hand, we can now prove Proposition~\ref{prop:cyclotomic-product-even}.

\begin{proof}[Proof of Proposition~\ref{prop:cyclotomic-product-even}]
Suppose that $p_h(t)$ is a product of cyclotomic polynomials.  Then Lemma~\ref{lem:10-multiples-of-4} says that $h \equiv 1\text{ or }2\pmod{5}$, so $h$ cannot be $3$, $4$, or $5$.  We will therefore require from now on that $h \geq 6$, so that $3h \leq 4h-6$ and hence the polynomial $q_h(t)$ from \eqref{eq:q_h} satisfies
\begin{align*}
q_h(t) &= t^{4h} - t^{4h-1} + 2t^{3h} + t^{2h} + 2t^{h} - t + 1 \\
&=  t^{4h} - t^{4h-1} + O(t^{4h-6}).
\end{align*}

By assumption, $p_h(t)$ has the form \eqref{eq:10-4}.  Following \eqref{eq:cyclotomic-q},
the polynomial $q_h(t)$ must then equal
\[ \Phi_5(t) \cdot \left(\prod_{j=1}^k \Phi_{n_j/2}(t)\right)^2 = t^{4h} - t^{4h-1} + O(t^{4h-6}). \]
The product being squared on the left is a monic polynomial with integer coefficients: if we write
\[ \prod_{j=1}^k \Phi_{n_j/2}(t) = t^{2h-2} + \sum_{i=1}^{2h-2} a_i t^{2h-2-i}, \]
then its square has the form
\[ \left(\prod_{j=1}^k \Phi_{n_j/2}(t)\right)^2 = t^{4h-4} + \sum_{i=1}^{4h-4} b_i t^{4h-4-i} \]
where $b_1 = 2a_1$ and 
\[ b_i = 2a_i + \sum_{j=1}^{i-1} a_ja_{i-j}, \qquad i\geq 2. \]
The summands on the right occur in pairs $a_ja_{i-j} = a_{i-j}a_j$ for $1 \leq j \leq \lfloor\frac{i-1}{2}\rfloor$, so it follows that
\[ b_i \equiv \begin{cases} 0, & i \text{\ odd} \\ (a_{i/2})^2 \equiv a_{i/2}, & i\text{\ even} \end{cases} \pmod{2}. \]
Multiplying by $\Phi_5(t) = t^4+t^3+t^2+t+1$, we see that the $t^{4h-1}$- and $t^{4h-2}$-coefficients of
\[ q_h(t) = (t^4+t^3+t^2+t+1)\left(t^{4h-4} + \sum_{i=1}^{4h-4} b_i t^{4h-4-i}\right) \]
are $-1$ and $0$ respectively, so that
\[ -1 = b_1+1 = 2a_1+1 \quad\Longrightarrow\quad a_1=-1, \]
and
\[ 0 = b_2+b_1+1 = (2a_2+a_1^2) +(2a_1)+1 \quad\Longrightarrow\quad a_2=0. \]
But then we also know that the $t^{4h-5}$-coefficient of $q_h(t)$ is $0$ by assumption, and yet it is also
\[ b_1+b_2+b_3+b_4+b_5 \equiv a_1+a_2 \equiv -1 \pmod{2}, \]
which is a contradiction.
\end{proof}

\section{HOMFLY homology of nearly fibered knots} \label{sec:nearly-fibered}

In this section we prove that reduced HOMFLY homology detects $5_2$ and each of the $P(-3,3,2n+1)$ pretzel knots, where $n\in\Z$.

As background, we recall that Khovanov and Rozansky \cite{khovanov-rozansky-1} defined for each integer $N \geq 1$ a bigraded $\mathfrak{sl}_N$ link homology $\bar{H}_N(K)$, where $\bar{H}_2(K)$ agrees with $\Khr(K)$ up to a change of grading.  In \cite{khovanov-rozansky-2} they also defined a triply graded homology theory $\bar{H}(K)$ whose graded Euler characteristic recovers the HOMFLY polynomial.  Rasmussen \cite{rasmussen-differentials} constructed for each $N \geq 1$ a spectral sequence
\begin{equation} \label{eq:rasmussen-ss}
\bar{H}(K) \ \Longrightarrow\ \bar{H}_N(K)
\end{equation}
which collapses for all large enough $N$. 

\begin{lemma} \label{lem:nf-dim-kh}
Suppose for some knot $K$ and some
\[ J \in \{5_2\} \cup \{ P(-3,3,2n+1) \mid n \in \Z \} \]
that
\[ \bar{H}(K;\Q) \cong \bar{H}(J;\Q) \]
as triply-graded vector spaces.  Then $\Delta_K(t) = \Delta_J(t)$, and
\[ \dim \bar{H}(K;\Q) = \dim \Khr(K;\Q) = \det(K). \]
\end{lemma}

\begin{proof}
The claim that $\Delta_K(t) = \Delta_J(t)$ follows from the fact that $\bar{H}(K)$ determines the HOMFLY polynomial of $K$, and hence its Alexander polynomial.  We note by taking $t=-1$ that this implies that $\det(K) = \det(J)$.

We now claim that
\begin{equation} \label{eq:dim-homfly}
\dim \bar{H}(K;\Q) = \dim \bar{H}(J;\Q) = \det(J) = \det(K).
\end{equation}
It suffices to prove the middle equality.  When $J = 5_2$ this follows from the fact that $J$ is a two-bridge knot: Rasmussen proved for each $N>4$ that $J$ is ``$N$-thin'' \cite[Theorem~1]{rasmussen-two-bridge}, and hence that the $\mathfrak{sl}_N$ homology $\bar{H}_N(J;\Q)$ has dimension $\det(J)$.  We take $N$ large enough so that the spectral sequence \eqref{eq:rasmussen-ss} collapses for $J$, and thus conclude that $\dim \bar{H}(J;\Q) = \det(J)$.  The case $J = P(-3,3,2n+1)$ is \cite[Lemma~9.1]{bs-nearly-fibered}, in which we use an identical argument for the two-bridge knot $P(-3,3,1) = 6_1$, and then we apply work of Wang \cite{wang-split} to get the general case.

Combining \eqref{eq:dim-homfly} with the case $N=2$ of \eqref{eq:rasmussen-ss} and Lemma~\ref{lem:det-k}, we now see that
\[ \det(K) = \dim \bar{H}(K;\Q) \geq \dim \bar{H}_2(K;\Q) = \dim \Khr(K;\Q) \geq \det(K). \]
Thus equality must hold throughout, completing the proof.
\end{proof}

Lemma~\ref{lem:nf-dim-kh} is enough to determine the instanton knot homology of such a knot $K$.

\begin{lemma} \label{lem:nf-khi}
Suppose for some knot $K$ and some
\[ J \in \{5_2\} \cup \{ P(-3,3,2n+1) \mid n \in \Z \} \]
that
\[ \bar{H}(K;\Q) \cong \bar{H}(J;\Q) \]
as triply-graded vector spaces.  Then $\dim \KHI(K) = \det(K)$, and $K$ has genus $1$ and
\[ \dim \KHI(K,1) = 2. \]
\end{lemma}

\begin{proof}
According to Lemma~\ref{lem:nf-dim-kh}, we have
\[ \Delta_K(t) = \Delta_J(t) = \begin{cases} 2t-3+2t^{-1}, & J=5_2 \\ -2t+5-2t^{-1}, & J=P(-3,3,2n+1) \end{cases} \]
and
\[ \dim \Khr(K;\Q) = \det(K) = \begin{cases} 7, & J=5_2 \\ 9, & J=P(-3,3,2n+1). \end{cases} \]
We once again apply Kronheimer and Mrowka's spectral sequence
\[ \Khr(K;\Q) \ \Longrightarrow\ I^\natural(\mirror{K};\Q) \]
of \cite[Proposition~1.2]{km-unknot}, together with the isomorphism
\[ I^\natural(\mirror{K};\Q) \cong \KHI(\mirror{K};\Q) \]
of \cite[Proposition~1.4]{km-unknot} and the invariance of $\dim \KHI$ under orientation reversal, to conclude that
\[ \dim \KHI(K;\Q) \leq \dim \Khr(K;\Q) = \det(K). \]
But the relation \eqref{eq:categorify} implies that if we write $\Delta_K(t) = \sum_i a_i t^i$ then
\[ \dim \KHI(K;\Q) \geq \sum_{i\in\Z} |a_i| \geq |\Delta_K(-1)| = \det(K), \]
so if $\dim \KHI(K;\Q) \leq \det(K)$ as well then each inequality must in fact be an equality.

In other words, we have shown that
\[ \dim \KHI(K,i;\Q) = |a_i| \]
for all $i\in\Z$, where the $a_i$ are the coefficients of $\Delta_J(t)$.  This is zero for all $i \geq 2$ and nonzero for $i=1$, so the genus detection property of $\KHI$ says that $g(K) = 1$.  And since $a_1 = \pm2$ we can also conclude that $\dim \KHI(K,1) = 2$, as claimed.
\end{proof}

Knots satisfying the conclusion of Lemma~\ref{lem:nf-khi} have been completely classified: we achieved the analogue of this with $\hfkhat$ in place of $\KHI$ in \cite{bs-nearly-fibered}, and then Li and Ye \cite{li-ye-nearly-fibered} showed that the same conclusion holds for $\KHI$.  Specifically, for any such knot the sutured complement of a genus-1 Seifert surface must be one of two possible sutured manifolds, up to orientation: for $\hfkhat$ this is \cite[Theorem~5.1]{bs-nearly-fibered}, and then \cite[Example~2.2]{li-ye-nearly-fibered} establishes the same result (with the same sutured manifolds) for $\KHI$.  Then \cite[Theorem~6.1]{bs-nearly-fibered} and \cite[Theorem~7.1]{bs-nearly-fibered} classify the knots that realize these sutured manifolds, by an argument that uses no Floer homology whatsoever.

\begin{proposition}[\cite{bs-nearly-fibered,li-ye-nearly-fibered}] \label{prop:nf-list}
Suppose that $K$ is a genus-1 knot, and that \[ \dim \KHI(K,1;\Q) = 2. \]  Then up to mirroring, either
\begin{itemize}
\item $\Delta_K(t) = 2t-3+2t^{-1}$, and $K$ is either $5_2$, $15n_{43522}$, or $16n_{696530}$; or
\item $\Delta_K(t) = -2t+5-2t^{-1}$, and $K$ is either some $P(-3,3,2n+1)$ or $15n_{115646}$,
\end{itemize}
where the knots $15n_{115646}$ and $16n_{696530}$ are in fact twisted Whitehead doubles $\mathrm{Wh}^\pm(T_{2,3},2)$ of the right-handed trefoil.
\end{proposition}

Using Shumakovitch's KhoHo program \cite{khoho}, we can compute that
\begin{equation} \label{eq:dim-kh-nf}
\dim \Khr(J;\Q) = \begin{cases} 17, & J = 15n_{43522} \\ 23, & J = 15n_{115646} \\ 25, & J = 16n_{696530}. \end{cases}
\end{equation}
With this information at hand, we are now ready to prove the remaining results from the introduction.

\begin{proof}[Proof of Theorem~\ref{thm:main-52}]
Suppose that $\bar{H}(K) \cong \bar{H}(5_2)$.  Then Lemmas~\ref{lem:nf-dim-kh} and \ref{lem:nf-khi} say that
\[ \Delta_K(t) = 2t - 3 + 2t^{-1}, \]
that $\dim \Khr(K;\Q) = 7$, and that $K$ has genus one, with $\dim \KHI(K,1;\Q) = 2$.  According to Proposition~\ref{prop:nf-list}, it follows that $K$ must be either $5_2$, $15n_{43522}$, or $16n_{696530}$ up to mirroring.  But the triple grading distinguishes $\bar{H}(5_2)$ from $\bar{H}(\mirror{5_2})$, since they have different signatures \cite[Corollary~5.1]{rasmussen-differentials}, so $K$ cannot be $\mirror{5_2}$.  It also cannot be either $15n_{43522}$ or $16n_{696530}$, or their mirrors, because by \eqref{eq:dim-kh-nf} we would then have $\dim \Khr(K;\Q) > 7$.  Thus $K$ must be $5_2$ after all.
\end{proof}

\begin{proof}[Proof of Theorem~\ref{thm:main-pretzel}]
Suppose that $\bar{H}(K) \cong \bar{H}(P(-3,3,2n+1))$ for some fixed $n$.  Then Lemma~\ref{lem:nf-dim-kh} says that
\[ \Delta_K(t) = -2t + 5 - 2t^{-1}, \]
so $\det(K) = 9$, and that
\[ \dim \bar{H}(K;\Q) = \dim \bar{H}_2(K;\Q) = \dim \Khr(K;\Q) = 9. \]
In particular, when $N=2$ the spectral sequence \eqref{eq:rasmussen-ss} collapses, and so by \cite[Theorem~1]{rasmussen-differentials} the bigrading on $\bar{H}_2(K)$ is completely determined by the triple grading on $\bar{H}(K)$.  This means that
\[ \Khr(K;\Q) \cong \Khr(P(-3,3,2n+1);\Q) \]
as bigraded vector spaces.

Moreover, by Lemma~\ref{lem:nf-khi} we know that $K$ has genus one, with $\dim \KHI(K,1;\Q) = 2$.  Proposition~\ref{prop:nf-list} says that $K$ must therefore be some pretzel knot $P(-3,3,2m+1)$, or $15n_{115646}$ or its mirror.  (We note that $P(-3,3,2m+1)$ is the mirror of $P(-3,3,-2m-1)$.)  But again by \eqref{eq:dim-kh-nf} the reduced Khovanov homologies of $15n_{115646}$ and its mirror are not 9-dimensional, so in fact $K$ cannot be either of these knots.  So now we have
\[ K = P(-3,3,2m+1) \]
for some $m\in \Z$, and all of these pretzel knots are distinguished by the bigradings on their Khovanov homologies, by \cite[Theorem~4.1]{starkston-pretzel} or more generally \cite[Theorem~3.2]{hedden-watson}.  Thus $m=n$ after all.
\end{proof}

\bibliographystyle{alpha}
\bibliography{References}

\end{document}